\documentclass[a4j,12pt]{article} 
\usepackage{amsmath}
\usepackage{amssymb}
\usepackage{amsfonts}
\usepackage{cases}
\usepackage{graphicx}
\usepackage{mathrsfs}
\usepackage{amsthm}
\usepackage{enumerate}
\usepackage[all]{xy} 
\usepackage{setspace}
\usepackage{latexsym}
\DeclareMathAlphabet{\mathpzc}{OT1}{pzc}{m}{it}
\newtheorem{defi}{Definition}[section]

\newtheorem{prop}[defi]{Proposition}
\newtheorem{coro}[defi]{Corollary}

\newtheorem{fact}[defi]{Fact}

 \title{A remark on the minimal dilation of the semigroup generated by a normal UCP-map}
 \author{Yusuke Sawada}
\date{}
\begin{document}
\maketitle

\noindent
Abstract: There are known three ways to construct the minimal dilation of the discrete semigroup generated by a normal unital completely positive map on a von Neumann algebra, which are given by Arveson, Bhat-Skeide and Muhly-Solel. In this paper, we clarify the relation of the constructions by Bhat-Skeide and Muhly-Solel.\\

\noindent
Keywords: dilations, completely positive maps, von Neumann algebras, $W^*$-bimodules, relative tensor products\\

\noindent
MCD (2010): 45L55 (Primary), 46L08 (Secondary)

\section{Introduction}
A dynamical transformation in a quantum physical system is described by a completely positive (CP) map on an operator algebra in a broad sense. We consider a von Neumann algebra $M$ acting on a Hilbert space $\mathcal{H}$ and a normal unital completely positive (UCP) map $T$ on $M$. The Stinespring's dilation theorem ensures the existence of a normal representation $(\pi,\mathcal{K})$ of $M$ and an isometry $v:\mathcal{H}\to\mathcal{K}$ such that $T(x)=v^*\pi(x)v$ for all $x\in
M$. When we consider a time evolution, the $n$-times transformation $T^n$ is important, but it is difficult to deal with representations $\{\pi_n\}_{n=1}^\infty$ associated with $\{T^n\}_{n=1}^\infty$. Now we consider the minimal dilation of the semigroup $\{T^n\}$ that is a larger von Neumann algebra $N\supset
M$ and a $*$-endomorphism $\alpha$ on $N$ such that $T^n$ is represented by $\alpha^n$ for each $n\in\mathbb{N}$, and it is hoped that $(N,\alpha)$ is minimal. To be accurate, the notion of minimal dilations is introduced in \cite{arve3} as the following.
\begin{defi}\label{1.1.}
Let $M$ be a von Neumann algebra and $T$ a normal UCP-map on $M$. A triplet $(N,\alpha,p)$ of a von Neumann algebra $N\supset
M$, a $*$-endomorphism $\alpha$ on $N$ and a projection $p\in
N$ is called a dilation of $T$ if $M=pNp$ and $T^n(x)=p\alpha^n(x)p$ for all $x\in
M$ and $n\in\mathbb{Z}_{\geq0}$. Moreover, a dilation $(N,\alpha,p)$ of $T$ is called minimal if $N$ is generated by $\bigcup_{n=0}^\infty\alpha^n(M)$ and the central support $c(p)$ of $p$ coincides with $1_N$.
\end{defi}

Dilations for a $C^*$-algebra $A$ and those for a continuous semigroup $\{T_t\}_{t\geq0}$ consisting of CP-maps on $A$ are also defined in a similar way. It is known that a minimal dilation is unique if it exists, see \cite{arve3}. That is, for a minimal dilation $(N,\alpha,p)$ of a normal UCP-map $T:M\to
M$, an operator $p\alpha^{n_1}(x_1)\cdots\alpha^{n_k}(x_k)p$ is uniquely determined by $T$, $n_1,\cdots,n_k\in\mathbb{Z}_{\geq0}$ and $x_1,\cdots,x_k\in
N$ for each $k\in\mathbb{N}$. Then the question of the existence of the minimal dilation arises. Bhat\cite{bhat1.5} proved the existence of the minimal dilation in the case when $A=\mathcal{B}(\mathcal{H})$ which consists of all bounded operators on a Hilbert space $\mathcal{H}$, and each $T_t$ is unital. In \cite{bhat2}, he generalized a way of the construction in stages and constructed a minimal dilation on a $C^*$-algebra $A$ under the assumption that $A$ is unital and $\|T_t(1_A)\|\leq1$ holds for all $t\geq0$. These are called the minimal dilation theory for $C^*$-algebras.

After that, Bhat-Skeide\cite{bhat-skei} constructed the minimal dilation on a von Neumann algebra $N\supset
A$ in the case when $A$ is a von Neumann algebra and a semigroup $\{T_t\}_{t\geq0}$ of normal CP-maps on $A$ has a continuity with respect to $t\geq0$, by using inductive limits of the tensor products of Hilbert bimodules. On the other hand, Arveson\cite{arve0.25},\cite{arve0.5} introduced the product systems and gave a one-to-one correspondence between product systems and semigroups $\{\alpha_t\}_{t\geq0}$ of $*$-endomorphisms called the $E_0$-semigroups. Consequently, he classified product systems. But after that, it is understood that Arveson's theory contains the dilation theory substantially, and his idea affected the constructions of dilations. Muhly-Solel\cite{muhl-sole} proved the result in \cite{bhat-skei} for normal UCP-maps $\{T_t\}_{t\geq0}$ by the similar way as in \cite{bhat-skei}. But the constructions are different in its appearance and no direct relation has not been established yet.

In this paper, we overview the constructions in \cite{bhat-skei} and \cite{muhl-sole}, of the minimal dilation in the case when given semigroup is a discrete semigroup $\{T^n\}_{n=0}^\infty$ generated by a normal UCP-map. We shall make their direct relationship clear and reveal that these constructions are essentially the same. The dilation of a discrete semigroup is applicable to the theory of non-commutative Poisson boundaries as revealed in \cite{izum}. 

In what follows, we assume that all Hilbert spaces are separable, and $\mathcal{B}(\mathcal{H},\mathcal{K})$ means the set of all bounded operators from $\mathcal{H}$ to $\mathcal{K}$. If $\mathcal{K}=\mathcal{H}$, we denote $\mathcal{B}(\mathcal{H},\mathcal{K})$ by $\mathcal{B}(\mathcal{H})$. For a set $X$, the identity map on $X$ is denoted by $\mbox{{\rm
id}}_X$ and $F^0=\mbox{{\rm
id}}_X$ for every map $F:X\to
X$. The unit of a unital algebra $A$ is denoted by $1_A$.

The author is deeply grateful to Prof. Shigeru Yamagami for insightful comments and suggestions.

\section{Preliminaries}
We recall the notion of $W^*$-modules and the related notations about them.
\begin{defi}
\begin{enumerate}[{\rm(1)}]
\item
For von Neumann algebras $N$ and $M$, a Hilbert space $\mathcal{H}$ with normal $*$-representations of $N$ and the opposite von Neumann algebra $M^\circ$ of $M$ is a $W^*$-$N$-$M$-bimodule if their representations commute. When $N=\mathbb{C}$ or $M=\mathbb{C}$, we call $\mathcal{H}$ a right $W^*$-$M$-module or a left $W^*$-$N$-module, respectively. We write a $W^*$-$N$-$M$-bimodule, a right $W^*$-$M$-module and a left $W^*$-$N$-module by ${}_N\mathcal{H}_M,\ \mathcal{H}_M$ and ${}_N\mathcal{H}$, respectively.
\item
Let $N$ be a von Neumann algebra, $X_N$ and $Y_N$ right $W^*$-$N$-modules, and ${}_NZ$ and ${}_NW$ left $W^*$-$N$-modules. $\mbox{\rm{Hom}}(X_N,Y_N)$ and $\mbox{\rm{Hom}}({}_NZ,{}_NW)$ are the sets of all right and left $N$-linear bounded maps, respectively. If $X=Y$ and $Z=W$, they are denoted by $\mbox{\rm{End}}(X_N)$ and $\mbox{\rm{End}}({}_NZ)$, respectively.
\item
We denote the standard representation space of a von Neumann algebra $M$ in {\rm\cite{haag1}} by $L^2(M)$.
\end{enumerate}
\end{defi}
We introduce the notion of Hilbert modules which are tools to construct the minimal dilation by the ways by Bhat-Skeide and Muhly-Solel. It is a module over a von Neumann algebra $M$ with an $M$-valued inner product.
\begin{defi}\label{2.1.}
Let $M$ be a von Neumann algebra and $E$ be right $M$-module. If a map $(\cdot,\cdot):E\times
E\to
M$ is defined and satisfies the following properties, then $E$ is called a Hilbert $M$-module.
\begin{enumerate}[{\rm(1)}]
\item
$(x,\alpha
y+\beta
z)=\alpha(x,y)+\beta(x,z)\hspace{5pt}(x,y,z\in
E,\ \alpha,\beta\in\mathbb{C}).$
\item
$(x,ya)=(x,y)a\hspace{5pt}(x,y\in
E,\ a\in
M).$
\item
$(x,y)^*=(y,x)\hspace{5pt}(x,y\in
E).$
\item
$(x,x)\geq0\hspace{5pt}(x\in
E).$
\item
For every $x\in
E$, $x=0$ if and only if $(x,x)=0$.
\item
$E$ is complete with respect to the norm defined by $\|x\|=\|(x,x)\|^\frac{1}{2}$.
\end{enumerate}

Suppose $E$ and $F$ are Hilbert $M$-modules. A right module homomorphism $b:E\to
F$ is called adjointable if there is a right module homomorphism $b^*:F\to
E$ called the adjoint of $b$ such that $(y,bx)=(b^*y,x)$ holds for every $x\in
E$ and $a\in
M$. We denote the set of all adjointable right module homomorphism by $\mathcal{B}^a(E,F)$. Automatically, $b\in\mathcal{B}^a(E,F)$ is bounded and $\mathcal{B}^a(E)=\mathcal{B}^a(E,E)$ is a $C^*$-algebra.

If there is a surjection $u\in\mathcal{B}^a(E,F)$ satisfying $(ux,uy)=(x,y)$ for every $x,y\in
E$, it is called an isomorphism or a unitary. Then $E$ and $F$ are said to be isomorphic and we write $E\cong
F$.
\end{defi}
\begin{defi}\label{2.2.}
Let $M$ and $N$ be von Neumann algebras and $E$ a Hilbert $N$-module. We call $E$ a Hilbert $M$-$N$-bimodule when it is an $M$-$N$-bimodule satisfying
\[
(x,ay)=(a^*x,y)
\]
for every $x,y\in
E$ and $a\in
M$.
\end{defi}
\begin{defi}\label{2.6.}
Let $M,N$ and $P$ be von Neumann algebras, $E$ a Hilbert $N$-$M$-bimodule and $F$ a Hilbert $M$-$P$-bimodule. Left and right actions of $a\in
M$ and $c\in
P$ on the algebraic tensor product $E\otimes_{{\rm
alg}}F$ are defined by $a(x\otimes
y)c=(ax)\otimes(yc)$ for each $x\in
E$ and $y\in
F$. We define that
\[
(x\otimes
y,x'\otimes
y')=(y,(x,x')y')
\]
for each $x,x'\in
E$ and $y,y'\in
F$, and put $\mathcal{N}=\{z\in
E\otimes_{\rm
alg}F\mid(z,z)=0\}$. The tensor product $E\otimes_MF$ of $E$ and $F$ is defined by the completion of $(E\otimes_{{\rm
alg}}F)/\mathcal{N}$ with respect to the norm induced from the above inner product. The left and right actions can be extended on $E\otimes_MF$, thus $E\otimes_MF$ becomes as Hilbert $N$-$P$-bimodules.
\end{defi}
The tensor product is associative, and for a Hilbert $M$-$M$-bimodule $E$, we regard as $\mathcal{B}^a(E)\subset\mathcal{B}^a(E)\otimes_M1_E\subset\mathcal{B}^a(E\otimes_ME).$

We introduce the GNS-construction with respect to a normal UCP-map, see \cite{skei} for example.
\begin{defi}
Suppose $M$ is a von Neumann algebra and $T:M\to
M$ is a normal UCP-map. We define a Hilbert $M$-$M$-bimodule $E(M,T)$ by the completion of $(M\otimes_{{\rm
alg}}M)/\mathcal{N}$ with respect to a norm induced from an inner product 
\[
(a\otimes
b,a'\otimes
b')_T=b^*T(a^*a')b'\hspace{20pt}(a,a',b,b'\in
M),
\]
where $\mathcal{N}=\{z\in
M\otimes_{\rm
alg}M\mid(z,z)_T=0\}$. If we put $\xi=1_M\otimes1_M+\mathcal{N}$, then ${\rm
span}(M\xi
M)$ is dense in $E(M,T)$ and $T(a)=(\xi,a\xi)$ holds for all $a\in
M$. We call the couple $(E(M,T),\xi)$ the GNS-representation with respect to $T$.
\end{defi}
There is an important identification in Bhat-Skeide's construction as the following.
\begin{defi}\label{2.7.}
Let $M$ be a von Neumann algebra acting on a Hilbert space $\mathcal{H}$ and $E$ a Hilbert $M$-module. Then $\mathcal{H}$ and $E$ are a Hilbert $M$-$\mathbb{C}$-bimodule and a Hilbert $\mathbb{C}$-$M$-bimodule, respectively, and hence we can define the tensor product $E\otimes_M\mathcal{H}$ as Hilbert bimodules. For $\xi\in
E$, we define $L_\xi:\mathcal{H}\ni
h\mapsto\xi\otimes
h\in
E\otimes_M\mathcal{H}$. Then we can identify $E$ as a right $M$-submodule of $\mathcal{B}(\mathcal{H},E\otimes_M\mathcal{H})$ by a map $:E\ni\xi\mapsto
L_\xi\in\mathcal{B}(\mathcal{H},E\otimes_M\mathcal{H})$. For $b\in\mathcal{B}^a(E)$, we can identify that $\mathcal{B}^a(E)\subset\mathcal{B}(E\otimes_M\mathcal{H})$ by
\[
b(\xi\otimes
h)=(b\xi)h\in
E\otimes_M\mathcal{H}\hspace{20pt}(\xi\in
E,\ h\in\mathcal{H}).
\]
If $E\subset\mathcal{B}(\mathcal{H},E\otimes_M\mathcal{H})$ is closed with respect to the strong operator topology, $E$ is called a von Neumann $M$-module.

Suppose $N$ is a von Neumann algebra. A von Neumann $M$-module $E$ is called a von Neumann $N$-$M$-bimodule if $E$ is a Hilbert $N$-$M$ bimodule, and a map $\rho:N\to\mathcal{B}(E\otimes_M\mathcal{H})$ defined by
\[
\rho(x)(\xi\otimes
h)=x\xi\otimes
h\hspace{20pt}(\xi\in
E,\ h\in\mathcal{H})
\]
is normal.

Then $\mathcal{B}^a(E)\subset\mathcal{B}(E\otimes_M\mathcal{H})$ is a von Neumann subalgebra; see \cite{skei}.
\end{defi}
A tensor product defined below is used in Muhly-Solele's construction.
\begin{defi}\label{4.1.}
Let $M$ be a von Neumann algebra acting on a Hilbert space $\mathcal{H}$ and $T$ a normal UCP-map on $M$. We define a sesquilinear form on the algebraic tensor product $M\otimes_{{\rm
alg}}\mathcal{H}$ by
\[
(x\otimes\xi,y\otimes\eta)=(\xi,T(x^*y)\eta)\hspace{20pt}(x,y\in
M,\ \xi,\eta\in\mathcal{H}).
\]
We define the Hilbert space $M\otimes_T\mathcal{H}=\overline{(M\otimes_{\rm
alg}\mathcal{H})/N}$, where $N=\{z\in
M\otimes_{{\rm
alg}}\mathcal{H}\mid(z,z)=0\}$.

A representation $\pi_T$ of $M$ on $M\otimes_T\mathcal{H}$ is defined by
\[
\pi_T(y)(x\otimes\xi)=yx\otimes\xi\hspace{20pt}(x\in
M,\ \xi\in\mathcal{H}).
\]
\end{defi}

\section{Some isomorphisms between $W^*$-bimodules}
In this section, some new results on isomorphisms between $W^*$-bimodules are presented as Proposition \ref{6.2.}--Corollary \ref{6.5.}. In Subsection 4.4, they will be used to see a relation between two constructions of the minimal dilation, which are given by Bhat-Skeide and Muhly-Solel.

First, we introduce notations with respect to $W^*$-modules and the relative tensor products in \cite{conn} and \cite{sauv}, and recall the facts about them (cf. \cite{take2} and \cite{bail-deni-have}). 
\begin{fact}\label{a}
\begin{enumerate}[{\rm(1)}]
\item
Let $M$ be a von Neumann algebra and ${}_M\mathcal{H}$ a $W^*$-$M$-module. For each positive normal functional $\phi$ on $M$, let $(\pi_\phi,\mathcal{H}_\phi,\xi_\phi)$ be the GNS-representation of $M$ with respect to $\phi$. We denote $\pi_\phi(x)\xi_\phi$ by $x\phi^\frac{1}{2}$ for each $x\in
M$. Since $\mathcal{H}$ is decomposable into cyclic representations, there exists a family of vectors $\{\xi_i\}_{i\in
I}$ in $\mathcal{H}$ such that $\mathcal{H}=\bigoplus_{i\in
I}\mathcal{H}_{\omega_i}$ where $\omega_i(x)=(\xi_i,x\xi_i)$. Moreover, if we denote the support of $\omega_i$ by $q_i$, we have 
\[
\mathcal{H}\cong\bigoplus_{i\in
I}(L^2(M)q_i)\cong(\bigoplus_{i\in
I}L^2(M))q
\]
as $W^*$-$M$-module where $q$ is the diagonal matrix whose diagonal entries are $\{q_i\}_{i\in
I}$.
\item
For a $W^*$-$M$-$N$-bimodule ${}_M\mathcal{H}_N$, we denote the dual Hilbert space of $\mathcal{H}$ by $\mathcal{H}^*$. For every $\xi^*\in
\mathcal{H}^*$, the right action of $x\in
M$ and the left action of $y\in
N$ to $\xi^*$ are defined by
\[
y\xi^*x=(x^*\xi
y^*)^*\in
\mathcal{H}^*.
\]
Then $\mathcal{H}^*$ becomes an $N$-$M$-bimodule.
\item
For each right $W^*$-$M$-module $\mathcal{H}_M$ and left $W^*$-$M$ module ${}_M\mathcal{K}$, we denote the relative tensor product of $\mathcal{H}$ and $\mathcal{K}$ with respect to $M$ by $\mathcal{H}\otimes^M\mathcal{K}$. The relative tensor product is associative. For a faithful semi-finite normal weight $\phi$, the subspace of sums of the form $\xi\phi^{-\frac{1}{2}}\eta$'s is dense in $\mathcal{H}\otimes^M\mathcal{K}$. Here, the notation $\xi\phi^{-\frac{1}{2}}\eta$ means that the tensor product of $\xi\in\mathcal{H}$ and a $\phi$-bounded vector $\eta\in\mathcal{K}$ For details, see {\rm\cite[Chapter 5, Appendix B]{conn}}. The relative tensor products have the following property for $W^*$-bimodule ${}_N\mathcal{H}_M$ and ${}_M\mathcal{K}_P$.
\begin{eqnarray*}
&&\mathcal{H}\otimes^ML^2(M)\cong\mathcal{H},\ L^2(M)\otimes^M\mathcal{K}\cong\mathcal{K},\\
&&\mathcal{K}\otimes^{(M')^\circ}\mathcal{K}^*\cong
L^2(M),\ \mathcal{K}^*\otimes^{M}\mathcal{K}\cong
L^2(M')
\end{eqnarray*}
where these isomorphisms mean as $W^*$-bimodules.
\item
We fix a von Neumann algebra $M$. Let $X_M$ be a Hilbert $M$-module and $\mathcal{H}_M$ be a right $W^*$-$M$-module. We can define the right $W^*$-module $\mathcal{H}(X)_M$ and the Hilbert $M$-module $X(\mathcal{H})_M$ as the following.
\begin{eqnarray*}
&&\mathcal{H}(X)_M=(X\otimes_ML^2(M))_M,\\
&&(x\otimes\xi,y\otimes\eta)_{\mathcal{H}(X)}=(\xi,(x,y)\eta)\hspace{20pt}(x\otimes\xi,y\otimes\eta\in\mathcal{H}(X)),\\
&&X(\mathcal{H})={\rm
Hom}(L^2(M)_M,\mathcal{H}_M)_M,\\
&&(x,y)_{X(\mathcal{H})}=x^*y\in{\rm
End}(L^2(M)_M)=M\hspace{20pt}(x,y\in
X(\mathcal{H})).
\end{eqnarray*}
This gives a one-to-one correspondence between Hilbert $M$-modules and right $W^*$-$M$-modules.
\end{enumerate}
\end{fact}
From now on, we fix a von Neumann algebra $M$ acting on a Hilbert space $\mathcal{H}$ and a normal UCP-map $T$ on $M$. We see relations between the relative tensor product $\otimes^M$ and the tensor product $\otimes_T$ defined in Section 1.
\begin{defi}
Since $M$ acts on the standard space $L^2(M)$ of $M$, we can define a left $W^*$-$M$-module $\mathcal{H}(M,T)=M\otimes_TL^2(M)$ {\rm(}Definition \ref{4.1.}{\rm)}. We define a right action of $M$ on $\mathcal{H}(M,T)$ by $(x\otimes\xi)y=x\otimes\xi
y$ for each $x,y\in
M$ and $\xi\in
L^2(M)$. Then $\mathcal{H}(M,T)$ is a $W^*$-$M$-$M$-bimodule.
\end{defi}
\begin{prop}\label{6.2.}
An isomorphism $\mathcal{H}(M,T)\otimes^M\mathcal{H}(M,T)\cong
M\otimes_T(M\otimes_TL^2(M))$ holds as $W^*$-bimodules.
\end{prop}
\begin{proof}
Let $\phi$ be a faithful normal state on $M$. We define a correspondence from an each vector
\begin{eqnarray*}
(x\otimes_Ty\phi^{\frac{1}{2}})\phi^{-\frac{1}{2}}(z\otimes_T\phi^{\frac{1}{2}}w)&\in&(M\otimes_TL^2(M))\otimes^\phi(M\otimes_TL^2(M))\\
&\cong&(M\otimes_TL^2(M))\otimes^M(M\otimes_TL^2(M))\\
&=&\mathcal{H}(M,T)\otimes^M\mathcal{H}(M,T)
\end{eqnarray*}
to a vector
\[
x\otimes_T((yz)\otimes_T(\phi^{\frac{1}{2}}w))\in
M\otimes_T(M\otimes_TL^2(M)).
\]
Then this correspondence gives a $W^*$-bimodule isomorphism.
\end{proof}
\begin{prop}\label{6.3.}
An isomorphism $\mathcal{H}(M,T)\otimes^M\mathcal{H}\cong
M\otimes_T\mathcal{H}$ holds as $W^*$-modules.
\end{prop}
\begin{proof}
Let $\phi$ be a faithful normal state on $M$. By Fact \ref{a} $(1)$ with respect to the decomposition of $\mathcal{H}$, each vector $\xi\in\mathcal{H}$ can be represented as $\bigoplus_{i\in
I}\xi_i$ for some $\xi_i\in
L^2(M)p_i$ and the projection $p_i$. We define a correspondence which maps
\[
(x\otimes_Ty\phi^\frac{1}{2})\phi^{-\frac{1}{2}}\bigoplus_{i\in
I}\xi_i\in(M\otimes_TL^2(M))\otimes^M\mathcal{H}
\]
to $x\otimes_T(\bigoplus_{i\in
I}y\xi_i)\in
M\otimes_T\mathcal{H}$. This correspondence is a unitary.
\end{proof}
Now, we have 
\begin{eqnarray*}
&&\mathcal{H}(M,T)\otimes^M\mathcal{H}(M,T)\otimes^M\mathcal{H}(M,T)\\
&&=(M\otimes_TL^2(M))\otimes^M(M\otimes_TL^2(M))\otimes^M(M\otimes_TL^2(M))\\
&&\cong(M\otimes_TL^2(M))\otimes^M(M\otimes_T(M\otimes_TL^2(M)))\\
&&\cong(M\otimes_T(M\otimes_T(M\otimes_TL^2(M)))).
\end{eqnarray*}
Indeed the first isomorphism is implied from Proposition \ref{6.2.} and the third isomorphism is given by a unitary defined by
\[
(x_1\otimes_Tx_2\phi^{\frac{1}{2}})\phi^{-\frac{1}{2}}(x_3\otimes_T(x_4\otimes_T\phi^{\frac{1}{2}}x_5))\mapsto
x_1\otimes_T((x_2x_3)\otimes_T(x_4\otimes_T\phi^{\frac{1}{2}}x_5))
\]
for each $x_1,x_2,x_3,x_4,x_5\in
M$ similarly to the proof of Proposition \ref{6.3.}. In the same way, we have
\[
\underbrace{\mathcal{H}(M,T)\otimes^M\cdots\otimes^M\mathcal{H}(M,T)}_{\mbox{{\rm
$n$} times}}\cong{}_M\underbrace{(M\otimes_T(\cdots\otimes_T(M}_{\mbox{{\rm
$n$ times}}}\otimes_TL^2(M))\cdots)).
\]
We define a $W^*$-$M$-$M$-bimodule
\[
{}_M\mathcal{H}_n(M,T)_M={}_M\underbrace{\mathcal{H}(M,T)\otimes^M\cdots\otimes^M\mathcal{H}(M,T)}_{\mbox{{\rm
$n$} times}}{}_M
\]
and a $W^*$-$(M')^\circ$-$(M')^\circ$-bimodule
\[
{}_{(M')^\circ}\mathcal{H}_n'(M,T)_{(M')^\circ}
={}_{(M')^\circ}\mathcal{H}^*\otimes^M\mathcal{H}_n(M,T)\otimes^{M}\mathcal{H}_{(M')^\circ}
\]
for each $n\in\mathbb{N}$.
\begin{prop}
We have an isomorphism
\begin{eqnarray*}
\mathcal{H}_n'(M,T)\cong\underbrace{\mathcal{H}_1'(M,T)\otimes^{(M')^\circ}\mathcal{H}_1'(M,T)}_{\mbox{$n$ times}}
\end{eqnarray*}
as $W^*$-bimodules for all $n\in\mathbb{N}$.
\end{prop}
\begin{proof}
By Fact \ref{a} $(3)$, we have isomorphisms
\begin{eqnarray*}
&&\mathcal{H}_1'(M,T)\otimes^{(M')^\circ}\mathcal{H}_1'(M,T)_{M'}\\
&&=\mathcal{H}^*\otimes^M\mathcal{H}(M,T)\otimes^M\mathcal{H}\otimes^{(M')^\circ}\mathcal{H}^*\otimes^M\mathcal{H}(M,T)\otimes^M\mathcal{H}\\
&&\cong\mathcal{H}^*\otimes^M\mathcal{H}(M,T)\otimes^ML^2(M)\otimes^M\mathcal{H}(M,T)\otimes^M\mathcal{H}\\
&&\cong\mathcal{H}^*\otimes^M\mathcal{H}(M,T)\otimes^M\mathcal{H}(M,T)\otimes^M\mathcal{H}\\
&&=\mathcal{H}_2'(M,T)
\end{eqnarray*}
as $W^*$-$(M')^\circ$-$(M')^\circ$-bimodules.
\end{proof}
\begin{coro}\label{6.5.}
We have an isomorphism
\begin{eqnarray*}
\mathcal{H}_n(M,T)\otimes^M\mathcal{H}&\cong&\underbrace{M\otimes_T(M\otimes_T\cdots(M\otimes_T(M}_{\mbox{$n$ times}}\otimes_T\mathcal{H}))\cdots).
\end{eqnarray*}
as $W^*$-modules for all $n\in\mathbb{N}$.
\end{coro}
Now, we have the following isomorphisms
\begin{eqnarray*}
&&\mbox{{\rm
Hom}}({}_M\mathcal{H},{}_M(M\otimes_TL^2(M))\otimes^M\mathcal{H})\\
&&\cong\mbox{{\rm
Hom}}({}_{(M')^\circ}\mathcal{H}^*\otimes^M\mathcal{H},{}_{(M')^\circ}\mathcal{H}^*\otimes^M(M\otimes_TL^2(M))\otimes^M\mathcal{H})\\
&&\cong\mbox{{\rm
Hom}}({}_{(M')^\circ}L^2(M'),{}_{(M')^\circ}\mathcal{H}^*\otimes^M(M\otimes_TL^2(M))\otimes^M\mathcal{H})\hspace{20pt}(\because\mbox{Fact \ref{a} }(3)).
\end{eqnarray*}
Then $\mbox{{\rm
Hom}}({}_{(M')^\circ}L^2(M'),{}_{(M')^\circ}\mathcal{H}^*\otimes^M(M\otimes_TL^2(M))\otimes^M\mathcal{H})$ corresponds to $\mathcal{H}^*\otimes^M(M\otimes_TL^2(M))\otimes^M\mathcal{H}=\mathcal{H}_1'(M,T)$ by Fact \ref{a} $(4)$.

\section{Two constructions of the minimal dilation}

In this section, we describe two constructions of the minimal dilation by Bhat-Skeide\cite{bhat-skei} and Muhly-Solel\cite{muhl-sole}, and see a relation between these constructions. We fix a von Neumann algebra $M$ acting on a Hilbert space $\mathcal{H}$ and a normal UCP-map $T$ on $M$.
\subsection{Bhat-Skeide's construction}

Let $(E(M,T),\xi)$ be the GNS-representation with respect to $T$. We put
\begin{eqnarray*}
&&E_n=\underbrace{E(M,T)\otimes_M\cdots\otimes_ME(M,T)}_{\mbox{$n$ times}},\\
&&\xi_n=\underbrace{\xi\otimes\cdots\otimes\xi}_{\mbox{$n$ times}}
\end{eqnarray*}
Then $(E_n, \xi_n)$ is the GNS-representation with respect to $T^n$ for each $n\in\mathbb{N}$ by the uniqueness of the GNS-representation. Let $E$ be an the inductive limit of the inductive system $(\{E_n\}_{n=0}^\infty,\{\xi_{n-m}\otimes{\rm
id}_{E_n}\}_{n,m=0}^\infty)$. We define $\mathcal{K}_n=E_n\otimes_M\mathcal{H}$ for each $n\in\mathbb{N}$ and $\mathcal{K}=E\otimes_M\mathcal{H}$. By the identification in Definition \ref{2.7.} and \cite{skei}, each $\overline{E}_n^s\subset\mathcal{B}(\mathcal{H},\mathcal{K}_n)$ is a von Neumann $M$-$M$-bimodule and $\overline{E}^s\subset\mathcal{B}(\mathcal{H},\mathcal{K})$ is so, where $\overline{\hspace{3pt}\cdot\hspace{3pt}}^s$ means the closure with respect to the strong operator topology under the embeddings.

We define an endomorphism $\theta$ on $\mathcal{B}^a(\overline{E}^s)$ by
\[
\theta(b)=b\otimes\mbox{{\rm
id}}_{\overline{E_1}^s}\in\mathcal{B}^a(\overline{\overline{E}^s\otimes\overline{E}_1^s}^s)\cong\mathcal{B}^a(\overline{E}^s)\hspace{20pt}(b\in\mathcal{B}^a(\overline{E}^s)).
\]
For each $a\in
M$, we define $j_0(a)\in\mathcal{B}^a(\overline{E}^s)$ by
\[
j_0(a)(\eta)=\xi
a(\xi,\eta)\hspace{20pt}(\eta\in
E)
\]
and $j_n=\theta^n\circ
j_0\in\mathcal{B}^a(\overline{E}^s)$ for each $n\in\mathbb{N}$. Then we have 
\[
j_m(1_M)j_n(a)j_m(1_M)=j_m(T^{n-m}(a))\]
for all $n\geq
m$ and $a\in
M$. We can identify that $M=j_0(M)$. Let $N$ be a von Neumann algebra generated by $j_{\mathbb{Z}_{\geq0}}(M)$, $p$ be $j_0(1_M)$ and $\alpha$ be a restriction of $\theta$ to $N$. Then the conditions in Definition \ref{1.1.} are satisfied.
\subsection{Muhly-Solel's construction}

Put $E(0)=M'$. For each $n\in\mathbb{N}$, we define $\mathcal{H}_n=\underbrace{(M\otimes_T(\cdots\otimes_T(M}_{\mbox{$n$ times}}\otimes_T\mathcal{H})\cdots))$ and $E(n)=\mbox{{\rm
Hom}}({}_M\mathcal{H},{}_M\mathcal{H}_n)$. Each $E(n)$ admits an $M'$-valued inner product defined by
\[
(X,Y)=X^*Y\in
M'\hspace{20pt}(X,Y\in
E(n)),
\]
and we can define left and right actions of $M'$ on $E(n)$ by
\begin{eqnarray*}
&&(xX)\xi=(\underbrace{1_M\otimes\cdots\otimes1_M}_{\mbox{$n$ times}}\otimes
x)X\xi\hspace{20pt}(x\in
M',\ X\in
E(n),\ \xi\in\mathcal{H}),\\
&&(Xx)\xi=X(x\xi)\hspace{20pt}(x\in
M',\ X\in
E(n),\ \xi\in\mathcal{H}).
\end{eqnarray*}
Then $E(n)$ becomes a $W^*$-correspondence over $M'$ in the sense of \cite{muhl-sole}, and we identify $E(n)\otimes_{M'}E(m)$ with $E(n+m)$ by a map
\[
U_{n,m}:E(n)\otimes_{M'}E(m)\ni
X_n\otimes
X_m\mapsto(\underbrace{1_M\otimes\cdots\otimes1_M}_{\mbox{$m$ times}}\otimes
X_n)X_m\in
E(n+m)
\]
for each $n,m\in\mathbb{Z}_{\geq0}$.

Now, we put $P_0=\mbox{{\rm
id}}_{E(0)}$ and $\mathcal{L}_0=\mathcal{H}$, and for each $n\in\mathbb{N}$ define a map $P_n:E(n)\to\mathcal{B}(\mathcal{H})$ by $P_n(X)=i^*\circ
X$ for each $X\in
E(n)$. Let $\mathcal{L}_n$ be a Hilbert space which is given by the completion of $E(n)\otimes_{{\rm
alg}}\mathcal{H}$ with respect to an inner product defined by
\[
(X\otimes\xi,Y\otimes\eta)=(X\xi,Y\eta)\hspace{20pt}(X,Y\in
E(n),\ \xi,\eta\in\mathcal{H}).
\]
For each $0<m<n$, we define isometric operators $u_{n,m}$ by
\begin{eqnarray*}
&&u_{n,m}=(U_{m,n-m}\otimes1_{\mathcal{B}(\mathcal{H})})({\rm
id}_{E(m)}\otimes\tilde{P}_{n-m}^*):\mathcal{L}_m\rightarrow\mathcal{L}_n,\\
&&u_{n,0}=\tilde{P}_n^*:\mathcal{L}_0\rightarrow\mathcal{L}_n,\\
&&u_{n,n}=1_{\mathcal{B}(\mathcal{L}_n)}:\mathcal{L}_n\rightarrow\mathcal{L}_n,
\end{eqnarray*}
where for all $Q:E(n)\to\mathcal{B}(\mathcal{H})$, a map $\tilde{Q}:\mathcal{L}_n\to\mathcal{H}_n$ is defined by $\tilde{Q}(X\otimes\xi)=Q(X)\xi$ for each $X\in
E(n)$ and $\xi\in\mathcal{H}$. Let $\mathcal{L}$ be the inductive limit of $(\{\mathcal{L}_n\}_{n=0}^\infty,\{u_{nm}\}_{n,m=0}^\infty)$ and $\iota_n:\mathcal{L}_n\rightarrow\mathcal{L}$ be the canonical embedding for each $n\in\mathbb{Z}_{\geq0}$. For each $m\in\mathbb{Z}_{\geq0}$ and $X_n\in
E(n)$, we define $V_n(X_n)\in\mathcal{B}(\mathcal{L})$ by
\begin{eqnarray*}
&&V_n(X_n)(\iota_m(X_m\otimes\xi))=\iota_{n+m}(U_{n,m}(X_n\otimes
X_m)\otimes\xi)\hspace{20pt}(X_m\in
E(m),\ \xi\in\mathcal{H}).
\end{eqnarray*}
We put $N=V_0(M')'$ and define $\alpha(x)=\tilde{V}_1({\rm
id}_{E(1)}\otimes
x)\tilde{V}_1^*$ for each $x\in
N$. Then $\alpha$ is a normal unital $*$-endomorphism on $N$ such that
\begin{eqnarray*}
&&\iota_0^*N\iota_0=M,\\
&&T^n(\iota_0^*x\iota_0)=\iota_0^*\alpha^n(x)\iota_0\hspace{20pt}(n\in\mathbb{Z}_{\geq0},\ x\in
N),\\
&&T^n(y)=\iota_0^*\alpha^n(\iota_0y\iota_0^*)\iota_0\hspace{20pt}(n\in\mathbb{Z}_{\geq0},\ y\in
M).
\end{eqnarray*}
We identify $M$ with $\iota_0M\iota_0^*$ and define a projection $p=\iota_0\iota_0^*$ in $N$. Then we have
\[
M\cong
\iota_0M\iota_0^*=\iota_0\iota_0^*N\iota_0\iota_0^*=pNp\subset
N.
\]
Thus the semigroup $\{\alpha^n\}_{n=0}^\infty$ is the minimal dilation of the semigroup $\{T^n\}_{n=0}^\infty$ in the sense of {\rm\cite{arve1}} and {\rm\cite{arve2}}. We have constructed the minimal dilation in the sense of Definition \ref{1.1.}.
\subsection{The minimal dilation on the standard space}
We simplify Muhly-Solel's construction of the minimal dilation when $\mathcal{H}=L^2(M)$. When we use the notation in Subsection $4.2$, $E(0)=M'$ and for each $n\in\mathbb{N}$,
\begin{eqnarray*}
&&\mathcal{H}_n=\underbrace{(M\otimes_T(\cdots\otimes_T(M}_{\mbox{$n$ times}}\otimes_TL^2(M))\cdots)),\\
&&E(n)=\mbox{{\rm
Hom}}({}_ML^2(M),{}_M\mathcal{H}_n),\\
&&\mathcal{L}_n=E(n)\otimes
L^2(M).
\end{eqnarray*}

Then for $n\in\mathbb{Z}_{\geq0}$, the map $U_n:E(n)\otimes
L^2(M)\ni
X\otimes\xi\mapsto
X\xi\in\mathcal{H}_n$ gives an isomorphism $\mathcal{L}_n\cong\mathcal{H}_n$ as Hilbert spaces. Now, for $n\geq
m$, we define an isometry
\[
v_{n,m}=U_nu_{n,m}U_m^*:\mathcal{H}_m\rightarrow\mathcal{H}_n
\]
where $u_{n,m}:\mathcal{L}_m\to\mathcal{L}_n$ is the isometry defined in Subsection $4.2$. Then $(\{\mathcal{H}_n\}_{n=0}^\infty,\{v_{nm}\}_{n,m=0}^\infty)$ is an inductive system, and let $\mathcal{H}'$ be the inductive limit of it. Similarly as Subsection $4.2$, for each $n\in\mathbb{Z}_{\geq0}$, let $\kappa_n:\mathcal{H}_n\to\mathcal{H}'$ be the canonical embedding and we define $V_n'(X_n)\in\mathcal{B}(\mathcal{H}')$ for each $X_n\in
E(n)$ by
\begin{eqnarray*}
&&V_n'(X_n)(\kappa_m(x_1\otimes\cdots\otimes
x_m\otimes\xi))=\kappa_{n+m}(x_1\otimes\cdots\otimes
x_m\otimes
X_n\xi)\\
&&\hspace{150pt}(m\in\mathbb{Z}_{\geq0},\ x_1\otimes\cdots\otimes
x_m\otimes\xi\in\mathcal{H}_m).
\end{eqnarray*}
Then we can prove an analogue of the result in Subsection $4.2$ by looking the proof of the original theorem (\cite{muhl-sole}) i.e., if we define
\begin{eqnarray*}
&&R=V_0'(M')',\\
&&\beta(x)=\tilde{V}_1'({\rm
id}_{E(1)}\otimes
x)\tilde{V}_1'^*\hspace{20pt}(x\in
R),
\end{eqnarray*}
then $\beta$ is a normal unital $*$-endomorphism on $R$ such that
\begin{eqnarray*}
&&\kappa_0^*R\kappa_0=M,\\
&&T^n(\kappa_0^*x\kappa_0)=\kappa_0^*\alpha^n(x)\kappa_0\hspace{20pt}(n\in\mathbb{Z}_{\geq0},\ x\in
N),\\
&&T^n(y)=\kappa_0^*\alpha^n(\kappa_0y\kappa_0^*)\kappa_0\hspace{20pt}(n\in\mathbb{Z}_{\geq0},\ y\in
M).
\end{eqnarray*}
\subsection{A relation between the two constructions}
In this subsection, we use the notations in Section 3, Subsection $4.1$ and $4.2$.

By Proposition \ref{6.3.}, 
\[
E(1)=\mbox{{\rm
Hom}}({}_M\mathcal{H},{}_MM\otimes_T\mathcal{H})\cong\mbox{{\rm
Hom}}({}_M\mathcal{H},{}_M(M\otimes_TL^2(M))\otimes^M\mathcal{H})
\]
holds, and hence $E(1)$ corresponds to $\mathcal{H}^*\otimes^M(M\otimes_TL^2(M))\otimes^M\mathcal{H}$. Hence we get a one-to-one correspondence 
\begin{eqnarray*}
&&E(n)\cong\underbrace{E(1)\otimes_{M'}\cdots\otimes_{M'}
E(1)}_{\mbox{{\rm
$n$ times}}}\longleftrightarrow\mathcal{H}^*\otimes^M\mathcal{H}_n(M,T)\otimes^M\mathcal{H}.
\end{eqnarray*}
for each $n\in\mathbb{N}$.

On the other hand, for each $n\in\mathbb{N}$, we can define the tensor product $\overline{E_n}^s\otimes_ML^2(M)$ as Definition 2.4 where $E_n$ is in Subsection $4.1$. Then we have $\overline{E_1}^s\otimes_ML^2(M)\cong
M\otimes_TL^2(M)=\mathcal{H}_1$ as left $W^*$-module. Indeed a map $:\overline{E_1}^s\otimes_ML^2(M)\ni(x\otimes_Ty)\otimes\xi\mapsto
x\otimes_Ty\xi\in\mathcal{H}_1$ gives an isomorphism because we have
\begin{eqnarray*}
(x_1\otimes_Ty_1\xi_1,x_2\otimes_Ty_2\xi_2)_{\mathcal{H}_1}&=&(y_1\xi_1,T(x_1^*x_2)y_2\xi_2)_{L^2(M)}\\
&=&(\xi_1,y_1^*T(x_1^*x_2)y_2\xi_2)_{L^2(M)}\\
&=&(\xi_1,(x_1\otimes_Ty_1,x_2\otimes_Ty_2)\xi_2)_{L^2(M)}\\
&=&((x_1\otimes_Ty_1)\otimes\xi_1,(x_2\otimes_Ty_2)\otimes\xi_2)
\end{eqnarray*}
for all $x_1,x_2,y_1,y_2\in
M$ and $\xi_1,\xi_2\in
L^2(M)$. By induction, we have
\[
\mathcal{H}_{n}(M,T)\cong
\overline{E_n}^s\otimes_ML^2(M)\hspace{20pt}
\]
for each $n\in\mathbb{N}.$ Thus we have a correspondence
\[
\overline{E_n}^s\longleftrightarrow\mathcal{H}_{n}(M,T).
\]

This concludes that the constructions of the dilation by Bhat-Skeide and Muhly-Solel are essentially the same.

\hspace{1pt}\\
\hspace{1pt}\\
\hspace{1pt}\\
\hspace{1pt}\\
\hspace{1pt}\\
\hspace{1pt}\\
\hspace{1pt}\\
Yusuke Sawada,\\
Graduate School of Mathematics,\\
Nagoya University,\\
Furocho, Chikusa-ku, Nagoya, 464-8602, Japan,\\
E-mail: m14017c@math.nagoya-u.ac.jp
\end{document}